\documentclass[12pt]{amsart}
\usepackage{amssymb}
\usepackage{bm}
\usepackage{graphicx}
\usepackage{psfrag}
\usepackage[usenames,dvipsnames]{xcolor}
\usepackage{enumerate}
\usepackage{multirow}
\usepackage{url}

\usepackage{subcaption}

\usepackage{comment}

\usepackage{algpseudocode}

\usepackage{mathtools}

\usepackage{xy}
\input xy
\xyoption{all}

\usepackage{tikz}

\newcommand{\excise}[1]{}

\newtheorem{thm}{Theorem}[section]
\newtheorem{lemma}[thm]{Lemma}

\newtheorem{cor}[thm]{Corollary}

\theoremstyle{definition}

\newtheorem{remark}[thm]{Remark}

\numberwithin{equation}{section}

\DeclareMathOperator\reg{reg} 
\DeclareMathOperator\height{ht} 

\newtheorem{innerthm}{Theorem}
\newenvironment{customthm}[1]{
  \renewcommand{\theinnerthm}{#1}
  \innerthm
}{
  \endinnerthm
}

\usepackage[backend=biber,style=numeric,  isbn=false, url=false, doi=false]{biblatex}
\usepackage{tikz-cd}

\begin{document}
\mbox{}
\title[Bounds for Regularity of Toric]{Bounds for Regularity of Toric Varieties in Terms of Multiplicity}
\author[Caviglia]{Giulio Caviglia}
\address{Department of Mathematics, Purdue University, 150 N. University Street, West Lafayette, IN 47907-2067, USA}
\email{gcavigli@purdue.edu}

\author[Hilmer]{Kieran Hilmer}
\address{Department of Mathematics, Purdue University, 150 N. University Street, West Lafayette, IN 47907-2067, USA}
\email{hilmerk@purdue.edu}


\date{\today}

\begin{abstract}
The Eisenbud-Goto conjecture stated that for a homogeneous prime ideal in a polynomial ring over an algebraically closed field, the Castelnuovo-Mumford regularity is bounded above by the multiplicity minus height plus one. In 2017, McCullough and Peeva showed that the regularity of an arbitrary homogeneous prime cannot be bounded by any polynomial function of the multiplicity. For ideals generated by polynomials with at most $c$ terms, we present a bound on the regularity which is less than triple exponential in $c$ and the multiplicity. Furthermore, for toric ideals we get a bound that is smaller than double exponential in terms of multiplicity.
\end{abstract}

\maketitle


\section{Introduction}
\label{sec:intro}

In 1984, Eisenbud and Goto \cite{Eisenbud_Goto_1984} conjectured that, for $P$ a homogeneous prime with no linear forms over an algebraically closed field, \[\reg(P) \leq \deg(P) - \text{height}(P) + 1,\] where $\reg(P)$ is the Castelnuovo-Mumford regularity of $P$ and $\deg(P)$ is the multiplicity of the coordiante ring of $P$ at the homogeneous maximal ideal. This longstanding open problem inspired much work in commutative algebra and algebraic geometry (\cite{Brodmann_Vogel_1993, Eisenbud_Ulrich_2011, Herzog_2003, Hoa_Miyazaki_1995, Lazarsfeld_1987, Pinkham_1986, Niu_2015, Ran_1990}).

In 2017, McCullough and Peeva (\cite{McCullough_Peeva_2017}) found counterexamples to the Eisenbud-Goto conjecture. Moreover, they constructed a family of primes $P_r$ such that the regularity of $P_r$ is exponential in the multiplicity $\deg(P_r)$, so that $\reg(P)$ is not bounded by any polynomial in $\deg(P)$. General bounds on $\reg(P)$ are notably bad. Indeed, the best such bound is given in \cite{caviglia2025explicitstillmanboundsdegrees}: \begin{equation}
    \reg(P) \leq (e^2) \uparrow^{e^4 - 1} (e + 2), \label{horrible bound}
\end{equation} where $e = \deg(P)$ and $\uparrow^{e^4-1}$ indicates $e^4-1$ copies of Knuth's up arrow. Knuth's up arrow can be used to define the Ackermann function, a notorious example of a function which grows faster than any function defined with for-loops and arithmetic operations (\cite[Ch. 12 \S 54]{Kleene_2005}).

Despite the conjecture being false in general, establishing when the inequality holds remains an important open problem. Eisenbud and Goto (\cite{Eisenbud_Goto_1984}) showed that arithmetically Cohen-Macaulay varieties satisfy the inequality, and \cite{Gruson_Lazarsfeld_Peskine_1983} showed that projective curves do as well. The authors in \cite{Lazarsfeld_1987} showed that the conjecture holds for smooth surfaces in characteristic zero. In \cite{Niu_2015}, the author proved the conjecture holds for normal surfaces with rational, Gorenstein elliptic, or log canonical singularities, and later \cite{NIU2022108320} showed a slightly worse bound for threefolds with rational singularities. Conversely, \cite{Choe_2024} constructed a family of counterexamples of any given dimension (at least 3) and height (at least 2) of the defining ideal. Recently, \cite{Han_2024} constructed counterexamples for projective singular surfaces. 

McCullough and Peeva (\cite{McCullough2021TheRC}) identified four ``interesting'' classes for which the conjecture may hold, of which one is projective \textit{toric} varieties (another was projective singular surfaces, for which there are now counterexamples). Given a $d \times n$ matrix $A$, the toric ideal $P_A$ of $A$ is generated by binomials which are in bijective correspondence with elements of $\ker A$. For toric varieties whose coordinate ring is simplicial with an isolated singularity, \cite{Herzog_2003} bounded the regularity by $(d-2)(n-1)$, which implies the conjecture. If the coordinate ring of a toric ideal is seminormal and simplicial, then it decomposes into a direct sum of monomial ideals with easily determined characteristics (\cite{NITSCHE2012345}). For height 2 toric ideals, Peeva and Sturmfels (\cite{Peeva_Sturmfels_1998}) proved the conjecture by reducing to the four variable case. 

The best bound for the regularity of a general toric variety in terms of multiplicity, however, is still \eqref{horrible bound}. In this work we provide a significantly better bound for a slightly more general class of varieties. Call an ideal $c$-nomial if it can be generated by polynomials with at most $c$ terms. We show that over an algebraically closed field, a homogeneous $c$-nomial prime ideal $P$ contains a maximal regular sequence in at most $b = c\frac{e(e^{e}-1)}{e-1}$ variables. Then the bound on regularity in terms of number of variables in \cite{cavigliadestefani2021} yields \[\reg(P) \leq (e-1)^{2^{b-2}+1}.\]  When we specialize to the toric case, we get \begin{customthm}{\ref{toricboundreg}}
    Over an algebraically closed field, a toric ideal $P_A$ with multiplicity $e$ satisfies $$\reg(P_A) \leq e^{e}(e-1).$$
\end{customthm}

\subsection*{Acknowledgments} The first named author was partially supported by a grant from the Simons Foundation (41000748, G.C.). The second named author was partially supported by Irena Swanson and by the Purdue Math Department.

\section{Background}
\label{sec: background}
 Let $k$ be an infinite field and $S$ a polynomial ring over $k$ in some number of variables. Let $P \subseteq S$ be a homogeneous prime ideal of height $h$. Let $e$ be the multiplicity of $S/P$. If $P$ is nondegenerate (i.e. contains no linear forms), then there is a maximal regular sequence in $P$ consisting of forms of degree $e$ (\cite[Lemma 7.2]{caviglia2025explicitstillmanboundsdegrees}). For such a regular sequence $f_1, \dots, f_h$, there is a degree $h(e-1)$ form $f$ such that $(f_1, \dots, f_h) : f$ equals $P$, as in \cite{chardin_2000}. If we assume that $k$ is algebraically closed, then $h \leq e-1$ (\cite[Proposition 5.3]{geometryofsyzygies}).

\begin{remark}
    The previous inequality is false if $k$ is not algebraically closed (\cite{caviglia2026quadraticlinearstrandsprime}). Let $R = \mathbb{R}[x_j, ix_j : 1 \leq j \leq h]$ with $i^2 = 1$ the imaginary unit. Let $S = \mathbb{R}[y_j, z_j : 1 \leq j \leq h]$ and let $P$ be the kernel of the map $S \to R$, $y_j \mapsto x_j$, $z_j \mapsto ix_j$. Then $P$ is a height $h$ homogeneous prime and the multiplicity of $S/P$ is $2$. Thus, the height can be arbitrarily larger than the multiplicity.
\end{remark}

Recall (\cite[Cor 2.13]{cavigliadestefani2021}) that if $\dim S = n \geq 3$ and $I \subseteq S$ is a homogeneous ideal generated in degree at most $D$, then \begin{equation}\label{pdreg}
    \reg(I) \leq D^{2^{n-2}}.
\end{equation}


\begin{remark}
    If we assume that $I = P$ is prime and $k$ is algebraically closed, then \eqref{pdreg} holds for $n = 1, 2$ for trivial reasons. If $n = 1$, then $P$ is either $(0)$ or a linear form. If $n = 2$, then $P$ is (0), principal, or the homogeneous maximal ideal.
\end{remark}

Let $A \in \mathbb{Z}^{d\times n}$ have column vectors $\mathbf{a_1}, \dots, \mathbf{a_n}$. The toric ring of $A$ is $$k[A] = k[t^\mathbf{a_j} : 1 \leq j \leq n] \subseteq k[t_i, t_i^{-1} : 1 \leq i \leq d],$$ and the toric ideal of $A$ is the kernel of the map $\varphi_A: k[x_1, \dots, x_n] \to k[A]$ by $x_j \mapsto t^\mathbf{a_j}$, which is denoted $P_A$. In \cite{Sturmfels_1996}, Sturmfels bounds the maximal degree of a primitive Gr\"obner basis element of a toric ideal. Using the Taylor resolution for the initial ideal of $P$ and upper semi-continuity, one gets \begin{equation} \label{regtoric}
    \reg(P_A) \leq \frac{n}{2}\cdot e\cdot h;
\end{equation} see  \cite[(7.2)]{Peeva_Sturmfels_1998}.

\section{Bounds for Regularity}
\label{sec:boundsupp}

Throughout this section, $k$ is an algebraically closed field, $S$ is a polynomial ring in an unknown number of variables, and $P \subseteq S$ is a nonzero homogeneous prime. Moreover, we will assume that for some fixed $c \in \mathbb{N}$, $P$ is a $c$-nomial ideal i.e. $P$ is generated by polynomials each of which is a $k$-linear combination of at most $c$ monomials. We fix $e \geq 2$ the multiplicity of $S/P$. Our goal is to bound $\reg(P)$ in terms of $e$.  In light of inequality $\eqref{pdreg}$, a bound on the number of variables in a minimal generating set of $P$ yields a bound on $\reg(P)$. This number of variables is bounded by the number of variables appearing in elements in a maximal regular sequence in $P$ because a generating set of $P$ can be algorithmically computed from its maximal regular sequnece using only those variables.
\begin{remark}
    A $c$-nomial ideal need not contain a regular sequence of polynomials with at most $c$ terms, nor the converse. For  
    example, if $\mathfrak{m} \subseteq k[x,y,z]$ is the homogeneous maximal ideal, then the ideal $(y^2-xz) + \mathfrak{m}^3$ contains the regular sequence $x^3, y^3, z^3$, but is not a monomial ideal under any change of variables.
    
    Conversely, toric ideals need not contain a maximal regular sequence consisting of 2-nomials. To see this, first notice that any 2-nomial in a toric ideal $P_A$ can be written as a scalar times a binomial $u - v$ for $u,v$ monomials. Because $P_A$ contains no monomials, if $a$ is a monomial and $a(u - v)$ is a 2-nomial appearing in a regular sequence, we can replace it with $u - v$. That is, any regular sequence of $2$-nomials can be replaced by one of the form $u_1 - v_1, \dots, u_l - v_l$ for $u_i, v_i$ coprime monomials. However, toric ideals of high enough codimension do not have maximal regular sequences of this form (\cite{Cattanicompleteintersections}). For example, the toric ideal of $$A = \begin{bmatrix}
        1 & 1 & \dots & 1\\
        1 & 2 & \dots & 14\\
        1 & 2^2 & \dots & 14^2\\
        \vdots& &\ddots  &\vdots\\
        1 & 2^{10} & \dots & 14^{10}
    \end{bmatrix}$$
    is height 3, but contains no binomial regular sequence of length 3, and hence no $2$-nomial regular sequence of length 3.
\end{remark}

However, the next result provides a characterization. Recall the avoidance lemma for vector spaces: let $k$ be an infinite field, $V$ a $k$-vector space with basis $B$, and $W_i$, $i \leq m$, vector spaces with $V \not\subseteq W_i$ for any $i$. Then a general linear combination of basis elements $b_i \in B \setminus W_i$ will avoid each $W_i$ i.e. there are $a_1, \dots, a_m \in k$ such that \begin{equation}
    \sum_{i=1}^m a_i b_i \in \left(V \setminus \bigcup W_i\right). \label{avoidancelemma}
\end{equation}

\begin{lemma}\label{controlledregseq}
    Let $S$ be a polynomial ring over an algebraically closed field $k$, and let $P \subseteq S$ a homogeneous prime ideal of multiplicity $e \geq 2$. Fix a minimal homogeneous generating set of $P$, and assume that the generators of degree at least 2 each have at most $c$ terms. Let $L$ be the sequence of linear forms, if any, appearing in that generating set. Then there exist variables $x_{i_1}, \dots, x_{i_b}$ with $b \leq c\frac{e(e^{e}-1)}{e-1}$ and a homogeneous maximal regular sequence $L, f_{1}, \dots, f_{h} \in P$ with each $f_i$ a degree $e$ form in $k[x_{i_1}, \dots, x_{i_b}]$.
\end{lemma}

\begin{proof}
    Let $l = |L|$. Let $\overline{P}$ be the image of $P$ in $S/L$ and let $h = \height(\overline{P})$. As $\overline{P}$ is a nondegenerate prime of the same multiplicity as $P$, $h \leq e-1$ and there is a maximal regular sequence of degree $e$ in $\overline{P}$, and hence in $P$ as well. We construct degree $e$ forms $f_1, \dots, f_h$ such that $L, f_{1}, \dots, f_h$ is a maximal regular sequence recursively.  Because $e \ge 2$, we have $h \geq 1$. Fix $B$ the set of degree $e$ forms obtained by multiplying the minimal generators of degree at least 2 by monomials which are not in $(L)$. If $h = 1$, then let $f_1$ be any element of $B$.
    
    Now, suppose that we have constructed $L, f_{1}, \dots, f_{m}$ with $m < h$. Let $P_1, \dots, P_s$ be the associated primes of $(L, f_{1}, \dots, f_{m})$. Because this is a homogeneous regular sequence, $P_1, \dots, P_s$ all have height $l + m$ (in particular, they are all minimal). The multiplicity of a regular sequence is given by the product of the degrees, which in this case is $e^m$, and each minimal prime contributes to the multiplicity according to the additivity formula (see \cite[Ch. V \S 2]{Serre_2000}). Thus, $s \leq e^m$. Let $V$ be the vector space of degree $e$ forms of $P$, and $W_i$ the space of degree $e$ forms of $P_i$. A basis for $V$ can be obtained from the set of generators of $P$ times monomials of the appropriate degree. Because $P$ has a regular sequence of length $h + l > m + l$ in degree $e$, we know that $V \not\subseteq W_i$. By the avoidance lemma \eqref{avoidancelemma}, there is a \begin{equation}
        f_{m+1} = \sum a_ib_i \in \left(P \setminus \bigcup_{i=1}^s P_s\right), \label{regseqrep}
    \end{equation} where $a_i \in k$ and each $b_i$ a basis element of $V$ avoiding $P_i$. Because each $P_i$ contains $L$, the set of monomials times $L$ is contained in $P_i$, so that each $b_i$ is not in $(L)$ i.e. it is in $B$. 

    We now consider the number of variables used in $f_1, \dots, f_h$, recalling that $h \leq e-1$. Each $f_{m}$ is a sum of at most $e^{m-1}$ polynomials of degree $e$ with support of size at most $c$. Thus, at most $ce^m$ variables occur in $f_{m}$. As such, the elements $f_{1}, \dots, f_{h}$ use at most $$ce + ce^2 + \dots + ce^{h} = c\frac{e(e^{h+1}-1)}{e-1} \leq c\frac{e(e^{e}-1)}{e-1}$$ variables. 
\end{proof}

\begin{remark}
    Suppose in Lemma \ref{controlledregseq} that $L = \emptyset$. Examine the new variables added by $f_{m+1}$ as in \eqref{regseqrep}. For each minimal prime $P_i$, we obtain a $c$-nomial $b_i \notin P_i$. If there is a term of some $b_j$ not divisible by any variable appearing in $f_1, \dots, f_m$, then $b_j$ will avoid all of the minimal primes, so we can take $f_{m+1} = b_j$. Otherwise, each term of each $b_i$ contains a variable already appearing in $f_1, \dots, f_m$. In this case $f_{m+1}$ is a sum of at most $e^{m}$ polynomials of degree $e$ with at most $c$ elements in the support, and each summand adds at most $c(e-1)$ variables. Thus, in the worst case at most $ce^m(e-1)$ new variables occur in $f_{m+1}$. As such, the elements $f_{1}, \dots, f_{h}$ use at most \begin{equation}\label{nondegenbound}
        ce + ce(e-1) + \dots + ce^{h-1}(e-1) = ce^{h} \leq ce^{e-1}
    \end{equation} variables. 
\end{remark}

\begin{cor}\label{pdbound}
    In the notation of Lemma \ref{controlledregseq}, if $P$ is nondegenerate, then the projective dimension of $P$ is at most  $ce^{e-1} - 1$.
\end{cor}

\begin{proof}
     Because $P$ is a minimal prime of $(f_{1}, \dots, f_h)$, we have $P \subseteq k[x_{i_1}, \dots, x_{i_b}]$. By  \eqref{nondegenbound}, $b \leq ce^{e-1}$. The result follows from Hilbert's Syzygy Theorem and the Auslander Buchbaum formula, noting that $P$ has nonzero depth.
\end{proof}

\begin{thm}
    Over an algebrically closed field, a homogeneosu $c$-nomial prime ideal $P$ with multiplicity $e$ satisfies \[\reg(P) \leq (e-1)^{2^{b-2}+1}\] where $b = c\frac{e(e^{e}-1)}{e-1}$ if $P$ is degenerate, or $b = ce^{e-1}$ if $P$ is nondegenerate.
\end{thm}

\begin{proof}
    Let $L, f_1, \dots, f_h$ be the maximal regular sequence in $P$ given in Lemma \ref{controlledregseq}. If an $l \in L$ uses a variable not appearing in the elements $f_1, \dots, f_h$, say $x_i$, then quotienting by $l$ amounts to substituting out the variable $x_i$, and so does not change the way the elements $f_1, \dots, f_h$ are written, nor the multiplicity or regularity of $P$. Therefore, we may assume that every form in $L$ uses only variables appearing in $f_1, \dots, f_h$ i.e. only the variables $x_{i_1}, \dots, x_{i_b}$. Under this assumption, $P \subseteq k[x_{i_1}, \dots, x_{i_b}]$. We now quotient by the remaining elements in $L$, and assume that $P \neq (0)$ is a nondegenerate prime in a polynomial ring $S$ in at most $b$ variables over $k$. Let $f$ be a degree $h(e-1)$ polynomial such that $P = (f_1, \dots, f_h) : f$, as in \cite{chardin_2000}. The short exact sequence \[0 \to S/P(-h(e-1)) \xrightarrow{\cdot f} S/(f_1, \dots, f_h) \to S/(f_1, \dots, f_h, f) \to 0\] yields that the regularity of $S/P(-h(e-1))$ is bounded above by the max of the regularity of $S/(f_1, \dots, f_h)$, which is $h(e-1)$, and the regularity of $S/(f_1, \dots, f_h, f)$ plus 1. By \eqref{pdreg} this at most $(h(e-1))^{2^{b-2}}.$ Thus, $h \leq (e-1)$ implies \[\reg(P) \leq  (h(e-1))^{2^{b-2}} - h(e-1) + 1 \leq (e-1)^{2^{b-2}+1}.\]
\end{proof}

When $P$ is toric, any linear form in $P$ is of the form $x_i - x_j$, and so by quotienting we may assume that $P$ is 2-nomial and nondegenerate. The bounds \eqref{regtoric} and \eqref{nondegenbound} yield the following corollary.

\begin{thm}\label{toricboundreg}
    Over an algebraically closed field, a homogeneous toric ideal $P_A$ with multiplicity $e$ satisfies $$\reg(P_A) \leq e^{e}(e-1).$$
\end{thm}


\printbibliography

\end{document}